\documentclass[11pt]{article}

\usepackage[margin=1.2in]{geometry}

\usepackage[T1]{fontenc}
\usepackage[utf8]{inputenc}
\usepackage[USenglish, UKenglish, english]{babel}
\usepackage{amssymb}
\usepackage{amsthm}
\usepackage{amsmath}
\usepackage[mathscr]{euscript}
\usepackage{enumitem}
\usepackage{graphicx}
\usepackage{MnSymbol}
 \let\mathscr\relax
\usepackage[scr]{rsfso}
\usepackage{tikz-cd}
\usepackage{tikz}
\usetikzlibrary{matrix,arrows,decorations.pathmorphing}
\usepackage{hyperref}
\usepackage{marginnote}
\usetikzlibrary{arrows.meta}

\usepackage{cases}

\usepackage{libertine}

\theoremstyle{definition}
\newtheorem{defin}{Definition}[section]
\theoremstyle{definition}

\theoremstyle{plain}
\newtheorem{theo}[defin]{Theorem}
\theoremstyle{plain}
\newtheorem{prop}[defin]{Proposition}
\theoremstyle{plain}
\newtheorem{lem}[defin]{Lemma}
\theoremstyle{plain}
\newtheorem{cor}[defin]{Corollary}
\theoremstyle{definition}
\newtheorem{rmk}[defin]{Remark}
\theoremstyle{definition}

\theoremstyle{definition}

\theoremstyle{plain}
\newtheorem{conj}[defin]{Conjecture}
\theoremstyle{definition}

\theoremstyle{definition}

\theoremstyle{plain}

\theoremstyle{plain}

\theoremstyle{plain}

\theoremstyle{definition}

\theoremstyle{plain}

\theoremstyle{definition}

\theoremstyle{definition}
\newtheorem*{defin*}{Definition}
\theoremstyle{definition}
\newtheorem*{ex*}{Example}
\theoremstyle{plain}
\newtheorem*{theo*}{Theorem}
\theoremstyle{plain}
\theoremstyle{plain}
\newtheorem*{conj*}{Conjecture}
\newtheorem*{prop*}{Proposition}
\theoremstyle{plain}
\newtheorem*{lem*}{Lemma}
\theoremstyle{plain}
\newtheorem*{cor*}{Corollary}
\theoremstyle{definition}
\newtheorem*{rmk*}{Remark}
\theoremstyle{definition}
\newtheorem*{exe*}{Exercise}

\theoremstyle{plain}
\newtheorem{theoA}{Theorem}

\theoremstyle{plain}

\theoremstyle{plain}

\theoremstyle{plain}

\theoremstyle{plain}
\newtheorem{corA}[theoA]{Corollary}

\numberwithin{equation}{section}

\usepackage{parskip}

\makeatletter
\def\thm@space@setup{%
  \thm@preskip=\parskip \thm@postskip=0pt
}
\makeatother

\setlist[enumerate]{label=(\roman*)}

\def\bl{{\rm bl}}

\def\h{{\rm ht}}

\def\irr{{\rm Irr}}

\def\aut{{\rm Aut}}

\def\n{{\mathbf{N}}} 

\def\c{{\mathbf{C}}} 
 
\def\z{{\mathbf{Z}}} 

\def\O{{\mathbf{O}}} 

\def\C{{\mathcal{C}}} 

\makeatletter
\newcommand{\uset}[3][0ex]{%
  \mathrel{\mathop{#3}\limits_{
    \vbox to#1{\kern-7\ex@
    \hbox{$\scriptstyle#2$}\vss}}}}
\makeatother

\newcommand{\wh}[1]{\widehat{#1}}

\newcommand{\ws}[1][1.5]{
  \mathrel{\scalebox{#1}[1]{$\sim$}}
}

\newcommand{\iso}[1]{\ws_{#1}}

\newcommand{\isoc}[1]{\ws_{#1}^c}

\usepackage{sectsty}
\allsectionsfont{\centering}

\usepackage{xparse,etoolbox}

\newcommand{\blocktheorem}[1]{%
  \csletcs{old#1}{#1}
  \csletcs{endold#1}{end#1}
  \RenewDocumentEnvironment{#1}{o}
    {\par\addvspace{1.5ex}
     \noindent\begin{minipage}{\textwidth}
     \IfNoValueTF{##1}
       {\csuse{old#1}}
       {\csuse{old#1}[##1]}}
    {\csuse{endold#1}
     \end{minipage}
     \par\addvspace{1.5ex}}
}

\blocktheorem{theoA}
\blocktheorem{conjA}
\blocktheorem{paraA}
\blocktheorem{corA}

\theoremstyle{plain}
\newtheorem*{iAMconj*}{The inductive Alperin--McKay Conjecture}

\theoremstyle{plain}
\newtheorem*{CTC*}{The Character Triple Conjecture}

\makeatletter
\def\blfootnote{\gdef\@thefnmark{}\@footnotetext}
\makeatother

\title{
{\huge\bf The Character Triple Conjecture for maximal defect characters and the prime $2$}\\
\author{\Large Damiano Rossi}
\blfootnote{\emph{$2010$ Mathematical Subject Classification:} $20$C$20$ ($20$C$15$)
\\
\emph{Key words and phrases:} Character Triple Conjecture, inductive Alperin--McKay condition, maximal defect characters, abelian defect groups
\\
\date{}
Large part of this work was carried out during a stay at the Mathematisches Forschungsinstitut Oberwolfach funded by a Leibniz Fellowship. I want to thank Gunter Malle for a careful reading of an earlier version of this paper and for providing useful comments.}}

\begin{document}

\renewcommand{\thetheoA}{\Alph{theoA}}

\renewcommand{\thecorA}{\Alph{corA}}

\selectlanguage{english}

\maketitle

\begin{abstract}
We prove that Sp\"ath's Character Triple Conjecture holds for every finite group with respect to maximal defect characters at the prime $2$. This is done by reducing the maximal defect case of the conjecture to the so-called inductive Alperin--McKay condition whose verification has recently been completed by Ruhstorfer for the prime $2$. As a consequence, we obtain the Character Triple Conjecture for all $2$-blocks with abelian defect groups by applying Brauer's Height Zero Conjecture, a proof of which is now available. We also obtain similar results for the block-free version of the Character Triple Conjecture at any prime $p$.
\end{abstract}

\section*{Introduction}

Based upon a large body of conjectural and computational evidence, the local-global principle in the representation theory of finite groups asserts that, given a prime number $p$ dividing the order of a finite group $G$, the representation theory of $G$ at the prime $p$ is largely determined by the $p$-local structure of the group. Here, the group $G$ plays the role of a global ambient and is opposed to the $p$-local structure which captures the embedding of the $p$-subgroups inside $G$. The questions arising in this context lead to some of the most important achievements in group representation theory of the past decades. Among others, we mention the proof of Brauer's Height Zero Conjecture from the $1950$s recently obtained in \cite{MNSFT}.

The conjectural evidence mentioned above consists of a series of statements that link different representation theoretic aspects of the group $G$ to its $p$-local structure. Apart from a few exceptions of a more structural flavour, all these statements can be ultimately reduced to proving the so-called Character Triple Conjecture for all quasi-simple groups. The latter, introduced by Sp\"ath in \cite{Spa17}, should be understood as the final result of an investigation initiated by Dade during the $1990$s that led to a sequence of increasingly stronger conjectures \cite{Dad92}, \cite{Dad94}, \cite{Dad97}. While relating global and local information through the notion of $p$-chains, an idea introduced by Robinson already in the $1980$s and subsequently exploited by Dade, Sp\"ath's conjecture provides a way to control fundamental cohomological and Clifford theoretical conditions that arise when considering the representation theoretical compatibility of normal group embeddings. This is achieved through the notion of $G$-block isomorphisms of character triples, hence the name of the conjecture. Given the technical nature of the Character Triple Conjecture, we refer the reader to Section \ref{sec:Preliminaries} for a precise definition.

The aim of this paper is to show that the Character Triple Conjecture holds at the prime $2$ for maximal defect characters. More precisely, we show that the conjecture holds for every Brauer $2$-block $B$ with respect to the non-negative integer $d=d(B)$ as specified in Remark \ref{rmk:CTC for max defect characters}.

\begin{theoA}
\label{thm:Main, CTC/DPC-max for p=2}
The Character Triple Conjecture holds for every Brauer $2$-block $B$ of a finite group with respect to the non-negative integer $d(B)$.
\end{theoA}

As an immediate consequence of Theorem \ref{thm:Main, CTC/DPC-max for p=2}, and using the \textit{if part} of Brauer's Height Zero Conjecture \cite{Kes-Mal13}, we deduce that the Character Triple Conjecture holds for all Brauer $2$-blocks with abelian defect groups.

\begin{corA}
\label{cor:Main CTC for abelian 2-blocks}
The Character Triple Conjecture holds for every Brauer $2$-block with abelian defect groups.
\end{corA}

The proofs of Theorem \ref{thm:Main, CTC/DPC-max for p=2} and Corollary \ref{cor:Main CTC for abelian 2-blocks} rely on the verification of the inductive Alperin--McKay condition introduced in \cite[Definition 7.2]{Spa13II} for the prime $2$ that was recently completed by Ruhstorfer in \cite{Ruh22AM}. In order to make use of this result, we prove a reduction theorem that shows the maximal defect case of the Character Triple Conjecture can be reduced to the verification of the inductive Alperin--McKay condition for all (covering groups of) non-abelian finite simple groups. In this paper, we will use the reformulation of this condition given in Conjecture \ref{conj:iAM} below. We can then state our reduction theorem as follows.

\begin{theoA}
\label{thm:Main, reduction for CTC-max}
Let $G$ be a finite group and $p$ a prime number. If every covering group of a non-abelian finite simple group involved in $G$ satisfies the inductive Alperin--McKay condition at the prime $p$, then the Character Triple Conjecture holds for every Brauer $p$-block $B$ of the group $G$ with respect to the non-negative integer $d(B)$.
\end{theoA}

While the above theorem appears to be new in nature, the reverse implication should be expected (at least among the experts in this research area). In fact, as mentioned above the Character Triple Conjecture implies most of the so-called local-global conjectures. Theorem \ref{thm:CTC-max implies iAM} below shows that the maximal defect case of the Character Triple Conjecture implies the inductive Alperin--McKay condition (as stated in Conjecture \ref{conj:iAM}). As a consequence, we deduce that these two statements are in fact logically equivalent.

The arguments used to prove the above results can be adapted to obtain analogous block-free statements. In particular, using the solution of the McKay Conjecture by Cabanes and Sp\"ath \cite{Cab-Spa-McK}, we are able to show that the block-free form of the Character Triple Conjecture holds at any prime $p$ for characters of degree coprime to $p$ (see Theorem \ref{thm:CTC block-free p=3}) and for every finite group with abelian Sylow $p$-subgroups (see Corollary \ref{cor:CTC block-free, abelian Sylow 3}). This will follow from a reduction of the block-free form of the Character Triple Conjecture to the verification of the inductive McKay condition for (the universal covering group of) non-abelian finite simple groups (see Theorem \ref{thm:CTC non-blockwise for maximal defect follows from iMcK}).

The paper is organised as follows: In Section \ref{sec:Preliminaries} we collect some background material and state the Character Triple Conjecture and the inductive Alperin--McKay condition. Section \ref{sec:iAM implies CTC} is devoted to the proof of Theorem \ref{thm:Main, reduction for CTC-max}. This is then used in Section \ref{sec:Theorem A ad Corollary B} in order to obtain Theorem \ref{thm:Main, CTC/DPC-max for p=2} and Corollary \ref{cor:Main CTC for abelian 2-blocks}. In Section \ref{sec:Converse, CTC implies iAM} we prove Theorem \ref{thm:CTC-max implies iAM}, a converse to Theorem \ref{thm:Main, reduction for CTC-max}. We conclude by sketching the proofs of the block-free analogues of all these results in Section \ref{sec:Block-free}

\section{Preliminaries and notation}
\label{sec:Preliminaries}

In this section, we collect some basic definitions and the statements of the conjectures considered below. Throughout this paper we freely use basic results from the representation theory of finite groups that can be found in standard texts such as \cite{Nag-Tsu89}, \cite{Nav98} but also the more recent \cite{Lin18I}, \cite{Lin18II}. We denote by $\irr(G)$ the set of complex valued irreducible characters of a finite group $G$. If $G$ is a normal subgroup of a larger group $A$ and $\chi$ is an irreducible character of $G$ fixed by the conjugacy action of $A$, then we say that $(A,G,\chi)$ is a character triple. We assume that the reader is familiar with the notion of $G$-block isomorphism, an equivalence relation on the set of character triples introduced in \cite[Definition 3.6]{Spa17} and denoted by $\iso{G}$.

For every prime number $p$, the set $\irr(G)$ admits a partition into the so-called Brauer $p$-blocks of $G$. Given a $p$-block $B$ of $G$, we denote by $\irr(B)$ the set of irreducible characters belonging to $B$. Conversely, given an irreducible character $\chi$ of $G$, we denote by $\bl(\chi)$ the unique $p$-block containing $\chi$. We will often suppress the $p$ from $p$-block and simply refer to $B$ as a block of $G$. Next, recall that for every $\chi\in\irr(G)$, the degree $\chi(1)$ of $\chi$ divides the order of $G$. We define the $p$-defect (or simply the defect) of $\chi$ as the non-negative integer $d(\chi)$ such that $p^{d(\chi)}=|G|_p/\chi(1)_p$ and where for every $n\geq 1$ we denote by $n_p$ the largest power of $p$ that divides $n$. If $d$ is a fixed non-negative integer, then $\irr^d(G)$ is the set of irreducible characters of $G$ of defect $d$ while, for a block $B$, we denote by $\irr^d(B)$ the intersection of $\irr^d(G)$ and $\irr(B)$. Next, to each block $B$ is associated a $G$-conjugacy class of $p$-subgroups $D$ of $G$ called the defect groups of $B$. If $|D|=p^m$, then we call $d(B):=m$ the defect of $B$. It is well known that $d(B)$ coincides with the maximum $d(\chi)$ for $\chi\in\irr(B)$. In particular, it follows that a character $\chi$ is of maximal defect in its block $B$ if and only if it is of height zero. Here the height of $\chi$ is defined as $\h(\chi):=d(B)-d(\chi)$ and for every $h\geq 0$ we denote by $\irr_h(B)$ the set of irreducible characters $\chi$ belonging to $B$ and with height $\h(\chi)=h$.

In order to state the Character Triple Conjecture, we need to introduce some more notation on $p$-chains. We refer here to \cite{Spa17} and \cite{Ros22}. Let $Z$ be a normal $p$-subgroup of $G$ and denote by $\mathfrak{N}(G,Z)$ the set of $p$-chains of $G$ starting with $Z$, that is the set of chains $\sigma=\{D_0=Z<D_1<\dots<D_n=D(\sigma)\}$ of $p$-subgroups $D_i$ of $G$ and where $D_0=Z$ and we denote by $D(\sigma)$ the final term of $\sigma$. The length of $\sigma$ is the number $|\sigma|=n$ of terms strictly containing $Z$. The reason for this convention stems from the fact that this definition of length coincides with the notion of dimension of $\sigma$ when viewed as a simplex (see, for instance, \cite[Section 1.1]{Ros-Homotopy}). We then obtain a partition of $\mathfrak{N}(G,Z)$ into the sets $\mathfrak{N}(G,Z)_\pm$ of $p$-chains $\sigma$ satisfying $(-1)^{|\sigma|}=\pm 1$. Since $Z$ is normal in $G$, the group $G$ acts by conjugation on $\mathfrak{N}(G,Z)$, by conjugating simultaneously each term of a $p$-chain $\sigma$, and we denote by $G_\sigma$ the stabiliser in $G$ of the chain $\sigma$, i.e. the intersection of the normalisers $\n_G(D_i)$ for each term $D_i$ of $\sigma$. Now, given a block $B$ of $G$ and a non-negative integer $d$, define $\C^d(B,Z)_\pm$ to be the set of pairs $(\sigma,\vartheta)$ where $\sigma$ is a $p$-chain belonging to $\mathfrak{N}(G,Z)_\pm$ and $\vartheta$ is an irreducible character of the stabiliser $G_\sigma$ with defect $d(\vartheta)=d$ and satisfying $\bl(\vartheta)^G=B$. Here, for a block $b$ of a subgroup $H$ of $G$, we denote by $b^G$ the block of $G$ obtained via Brauer induction whenever it is defined \cite[Section 4]{Nav98}. The set of such characters is often denoted by $\irr^d(B_\sigma)$. Since the action of $G$ fixes $B$ and $Z$, the group $G$ acts on $\C^d(B,Z)_\pm$. We denote by $\overline{(\sigma,\vartheta)}$ the $G$-orbit of $(\sigma,\vartheta)\in\C^d(B,Z)_\pm$ and by $\C^d(B,Z)_\pm/G$ the corresponding set of $G$-orbits. We can now state the Character Triple Conjecture in the form introduced by Sp\"ath in \cite[Conjecture 6.3]{Spa17}.

\begin{conj}[Character Triple Conjecture]
\label{conj:CTC}
Let $G\unlhd A$ be finite groups, $p$ a prime number, and assume that $\O_p(G)$, the largest normal $p$-subgroup of $G$, is contained in the centre of $G$. Then, for every $p$-block $B$ of $G$ with non-central defect groups and every non-negative integer $d$, there exists an $A_B$-equivariant bijection
\[\Omega:\C^d(B,\O_p(G))_+/G\to\C^d(B,\O_p(G))_-/G\]
such that
\[\left(A_{\sigma,\vartheta},G_\sigma,\vartheta\right)\iso{G}\left(A_{\rho,\chi},G_\rho,\chi\right)\]
for every $(\sigma,\vartheta)\in\C^d(B,\O_p(G))_+$ and $(\rho,\chi)\in\Omega(\overline{(\sigma,\vartheta)})$.
\end{conj}

\begin{rmk}
\label{rmk:CTC for max defect characters}
We say that the Character Triple Conjecture holds at the prime $p$ for maximal defect characters, if Conjecture \ref{conj:CTC} holds at the prime $p$ for every $p$-block $B$ of a finite group and with respect to $d=d(B)$.
\end{rmk}

Observe that the $G$-block isomorphism of character triples considered in the statement above does not depend on the choice of representatives $(\sigma,\vartheta)$ and $(\rho,\chi)$ in the corresponding $G$-orbits thanks to \cite[Lemma 3.8 (c)]{Spa17}. Moreover, notice that the assumption on $\O_p(G)$ is not restrictive. In fact, we could replace $\O_p(G)$ with any central $p$-subgroup $Z$ of $G$ and consider blocks $B$ with defect groups strictly containing $Z$ (see \cite[Conjecture 2.2]{Ros22}). However, in this case $Z\leq \O_p(G)$ and the result follows trivially whenever $Z\neq \O_p(G)$ thanks to a well-known contractibility argument due to Quillen \cite[Lemma 2.3]{Ros22}. In Section \ref{sec:iAM implies CTC} we will consider the case where $Z$ is not required to be contained in the centre of $G$. The equivalence of this latter form with Conjecture \ref{conj:CTC} above is however not immediate to prove (this will appear in a future work of the author \cite{Ros-Reduction_CTC}).

\begin{rmk}
In some of the arguments given in Section \ref{sec:iAM implies CTC} it will be useful to consider normal $p$-chains. A $p$-chain $\sigma$ is said to be normal if each term $D_i$ is normal in the largest term $D(\sigma)$. Proceeding as in the proof of \cite[Proposition 6.10]{Spa17}, and following previous ideas introduced by Kn\"orr and Robinson (see \cite[Proposition 3.3]{Kno-Rob89}), it follows that when dealing with Conjecture \ref{conj:CTC} it is no loss of generality to assume that each $p$-chain considered in the definition of $\C^d(B,\O_p(G))_\pm$ is normal. For these reasons, we will keep using normal $p$-chains throughout the rest of the paper without further reference. This approach was also used in \cite{Ros22} without any comment.
\end{rmk}

We recall that Conjecture \ref{conj:CTC} implies Dade's Extended Projective Conjecture \cite[4.10]{Dad97} according to \cite[Proposition 6.4]{Spa17} and, as mentioned already in the introduction, that it should be understood as an analogue of the final Dade's Inductive Conjecture \cite[5.8]{Dad97}. In fact, it was announced long ago that the latter would reduce to quasi-simple groups although a proof of this result has not yet appeared. Nevertheless, it was shown in \cite[Theorem 1.3]{Spa17} that if the Character Triple Conjecture holds for quasi-simple groups then the weaker Dade's Projective Conjecture holds for every finite group. A final reduction of the Character Triple Conjecture to quasi-simple groups has recently been completed in \cite{Ros-Reduction_CTC}. Regarding the state of the art of the Character Triple Conjecture, we refer the reader to \cite[Section 9]{Spa17} for the case of sporadic groups, special linear groups of degree $2$, and blocks with cyclic defect, to \cite{Ros22} for the case of $p$-solvable groups, and to the series of papers \cite{Ros-Generalized_HC_theory_for_Dade}, \cite{Ros-Clifford_automorphisms_HC}, \cite{Ros-Unip}, and \cite{Ros-Homotopy} for the case of finite simple groups of Lie type in non-defining characteristic.

Next, we consider the inductive Alperin--McKay condition. In its most popular form, this condition is stated for simple groups and their covering groups (see \cite[Definition 7.2]{Spa13II}). Nevertheless, this condition can be stated for every finite group. In this paper, we consider the following form in which the cohomological and Clifford theoretic requirements are reformulated in terms of $G$-block isomorphisms of character triples.

\begin{conj}[inductive Alperin-McKay condition]
\label{conj:iAM}
Let $G\unlhd A$ be finite groups, $p$ a prime number, and consider a $p$-block $B$ of $G$ with defect group $D$ and Brauer correspondent $b$ in $\n_G(D)$. Then there exists an $\n_A(D)_B$-equivariant bijection
\[\Theta:\irr_0(B)\to\irr_0(b)\]
such that
\[\left(A_\vartheta,G,\chi\right)\iso{G}\left(\n_A(D)_\vartheta, \n_G(D),\Theta(\vartheta)\right),\]
for every $\vartheta\in\irr_0(B)$.
\end{conj}

Observe that the condition on character triples in Conjecture \ref{conj:iAM} could equivalently be stated by using the relation $\geq_b$ considered in \cite{Spa18}. Moreover, we point out that, arguing as in the proof of \cite[Proposition 6.8]{Spa17}, it follows that the inductive Alperin--McKay condition from \cite[Definition 4.12]{Spa18} holds for the universal covering group $X$ of a non-abelian simple group $S$ if and only if Conjecture \ref{conj:iAM} holds for every quasi-simple group $Y$ covering $S$ with respect to $Y\unlhd Y\rtimes \aut(Y)$. Then, \cite[Theorem C]{Spa13II} can be restated by saying that if Conjecture \ref{conj:iAM} holds for every quasi-simple group, then the Alperin--McKay Conjecture holds for every finite group. Finally, a much stronger version of this reduction theorem was obtained in \cite[Theorem 7.1]{Nav-Spa14I} where the authors showed that Conjecture \ref{conj:iAM} reduces to quasi-simple groups.

\section{Proof of Theorem \ref{thm:Main, reduction for CTC-max}}
\label{sec:iAM implies CTC}

In order to prove Theorem \ref{thm:Main, reduction for CTC-max}, we need the following slightly stronger statement in which we allow the $p$-subgroup $Z$ from Conjecture \ref{conj:CTC} to be non-central. Recall that a group $S$ is said to be involved in $G$ if there exist subgroups $K\unlhd H\leq G$ such that $S$ is isomorphic to $H/K$.

\begin{theo}
\label{thm:CTC for maximal defect follows from iAM}
Let $G$ be a finite group, consider a prime $p$, and suppose that the inductive Alperin-McKay condition (as stated in Conjecture \ref{conj:iAM}) holds at the prime $p$ for every covering group of a non-abelian finite simple group involved in $G$. Let $G\unlhd A$ and $U\unlhd G$ a $p$-subgroup of order $|U|=p^m$. Then, for every $p$-block $B$ of $G$ with defect $d:=d(B)>m$ there exists an $\n_A(U)_B$-equivariant bijection
\[\Omega_{B,U}:\C^d(B,U)_+/G\to \C^d(B,U)_-/G\]
such that
\[\left(A_{\sigma,\vartheta},G_\sigma,\vartheta\right)\iso{G}\left(A_{\rho,\chi},G_\rho,\chi\right)\]
for every $(\sigma,\vartheta)\in\C^d(B,U)_+$ and $(\rho,\chi)\in\Omega_{B,U}(\overline{(\sigma,\vartheta)})$.
\end{theo}

We now prove the above theorem by induction on the order of $G$ and assume that the result holds for every choice of groups $U'\unlhd G'\unlhd A'$ with $|G'|<|G|$. We proceed by proving a series of intermediate results. In what follows, given a normal $p$-subgroup $Q$ of a finite group $H$ and a collection $\mathcal{B}$ of $p$-blocks of $H$, we define the set of pairs
\[\C^f(\mathcal{B},Q):=\coprod\limits_{b\in\mathcal{B}}\C^f(b,Q)\]
for any non-negative integer $f$. Notice that the partition of each set $\C^f(b,Q)$ into $\C^f(b,Q)_\pm$ induces a partition of the union $\C^f(\mathcal{B},Q)$ into the naturally defined subsets $\C^f(\mathcal{B},Q)_\pm$.

\begin{lem}
\label{lem:CTC for maximal defect follows from iAM, 1}
Let $Q$ be a $p$-subgroup of $G$ satisfying $U<Q<D$ for some defect group $D$ of $B$ and denote by $B_Q$ the set of those $p$-blocks $b$ of $\n_G(Q)$ satisfying $b^G=B$ and $d(b)=d$, where $|D|=p^d$. Then there exists an $\n_A(Q)_B$-equivariant bijection
\[\Omega_{B_Q,Q}:\C^d(B_Q,Q)_+/\n_G(Q)\to\C^d(B_Q,Q)_-/\n_G(Q)\]
such that
\[\left(\n_A(Q)_{\varsigma,\varphi},\n_G(Q)_\varsigma,\varphi\right)\iso{\n_G(Q)}\left(\n_A(Q)_{\varrho,\psi},\n_G(Q)_\varrho,\psi\right)\]
for every $(\varsigma,\varphi)\in\C^d(B_Q,Q)_+$ and $(\varrho,\psi)\in\Omega_{B_Q,Q}(\overline{(\varsigma,\varphi)})$ where we now denote by $\overline{(\varsigma,\varphi)}$ the $\n_G(Q)$-orbit of $(\varsigma,\varphi)$.
\end{lem}

\begin{proof}
Without loss of generality we may assume that $U=\O_p(G)$. For if it weren't, the argument used in the proof of \cite[Lemma 2.3]{Ros22} would give the bijection required in Theorem \ref{thm:CTC for maximal defect follows from iAM}. In particular, the assumption $U<Q$ implies that $\n_G(Q)<G$ and therefore the statement of Theorem \ref{thm:CTC for maximal defect follows from iAM} holds true for $Q\unlhd \n_G(Q)\unlhd \n_A(Q)$. Then, if $b$ is any block belonging to $B_Q$ and $|Q|=p^l$, the assumption $Q<D$ implies that $d=d(b)>l$ and we obtain an $\n_A(Q)_b$-equivariant bijection
\[\Omega_{b,Q}:\C^d(b,Q)_+/\n_G(Q)\to\C^d(b,Q)_-/\n_G(Q)\]
such that
\begin{equation}
\label{eq:CTC for maximal defect follows from iAM, 1}
\left(\n_A(Q)_{\varsigma,\varphi},\n_G(Q)_\varsigma,\varphi\right)\iso{\n_G(Q)}\left(\n_A(Q)_{\varrho,\psi},\n_G(Q)_\varrho,\psi\right)
\end{equation}
for every $(\varsigma,\varphi)\in\C^d(b,Q)_+$ and $(\varrho,\psi)\in\Omega_{b,Q}(\overline{(\varsigma,\varphi)})$. Next, observe that $\n_A(Q)_B$ acts by conjugation on the set of blocks $B_Q$ and choose an $\n_A(Q)_B$-transversal $\mathcal{S}$ in $B_Q$. For each block $b\in\mathcal{S}$, notice that $\n_A(Q)_b\leq \n_A(Q)_B$ and fix an $\n_A(Q)_b$-transversal $\mathcal{S}_b^+$ in $\C^d(b,Q)_+/\n_G(Q)$. Since the bijection $\Omega_{b,Q}$ is $\n_A(Q)_b$-equivariant, we deduce that the image $\mathcal{S}_b^-$ of $\mathcal{S}_b^+$ under the map $\Omega_{b,Q}$ is an $\n_A(Q)_b$-transversal in $\C^d(b,Q)_-/\n_G(Q)$. It follows that the set
\[\mathcal{T}:=\coprod\limits_{b\in \mathcal{S}}\mathcal{S}_b^\pm\]
is an $\n_A(Q)_B$-transversal in $\C^d(B_Q,Q)_\pm/\n_G(Q)$ and that the maps $\Omega_{b,Q}$, for $b\in\mathcal{S}$, induce a bijection between the transversals $\mathcal{T}^+$ and $\mathcal{T}^-$. This bijection can be extended to an $\n_A(Q)_B$-equivariant bijection $\Omega_{B_Q,Q}$ between $\C^d(B_Q,Q)_+/\n_G(Q)$ and $\C^d(B_Q,Q)_-/\n_G(Q)$ by setting
\[\Omega_{B_Q,Q}\left(\overline{(\varsigma,\varphi)}^x\right):=\overline{(\varrho,\psi)}^x\]
for every $x\in\n_A(Q)_B$ and every $\overline{(\varsigma,\varphi)}\in\mathcal{T}^+$ corresponding to $\overline{(\varrho,\psi)}\in\mathcal{T}^-$. Furthermore, observe that the $\n_G(Q)$-isomorphism required in the statement is the same as the one given in \eqref{eq:CTC for maximal defect follows from iAM, 1} by the bijections $\Omega_{b,Q}$. This completes the proof.
\end{proof}

Before proceeding to the next step, we introduced some further notation. For every $p$-subgroup $Q$ of $G$ strictly containing $U$, we define the subset $\C^d_Q(B,U)$ consisting of those pairs $(\sigma,\vartheta)$ in $\C^d(B,U)$ such that the $p$-chain $\sigma$ satisfies $\sigma=\{D_0=U<D_1=Q<D_2<\dots<D_n\}$ for some $n\geq 1$. In other words, $\C^d_Q(B,U)$ is the set of pairs $(\sigma,\vartheta)$ such that $Q$ is the second term of the chain $\sigma$. In this case, we also define $\C^d_Q(B,U)_\pm$ as the intersection of $\C^d_Q(B,U)$ with $\C^d(B,U)_\pm$. If we denote by $\n_A(U,Q)$ the intersection $\n_A(U)\cap \n_A(Q)$, then $\n_A(U,Q)_B$ acts by conjugation on the sets $\C^d_Q(B,U)_\pm$. Using Lemma \ref{lem:CTC for maximal defect follows from iAM, 1}, we can construct a bijection between the sets $\C^d_Q(U,B)_\pm$.

\begin{cor}
\label{cor:CTC for maximal defect follows from iAM, 2}
Let $Q$ be a $p$-subgroup of $G$ satisfying $U<Q<D$ for some defect group $D$ of $B$ and set $d:=d(B)$. Then there exists an $\n_A(U,Q)_B$-equivariant bijection
\[\Theta_Q:\C^d_Q(B,U)_+/\n_G(Q)\to\C^d_Q(B,U)_-/\n_G(Q)\]
such that
\[\left(\n_A(Q)_{\sigma,\vartheta},\n_G(Q)_\sigma,\vartheta\right)\iso{\n_G(Q)}\left(\n_A(Q)_{\rho,\chi},\n_G(Q)_\rho,\chi\right)\]
for every $(\sigma,\vartheta)\in\C^d_Q(B,U)_+$ and $(\rho,\chi)\in\Theta_Q(\overline{(\sigma,\vartheta)})$ where we now denote by $\overline{(\sigma,\vartheta)}$ the $\n_G(Q)$-orbit of $(\sigma,\vartheta)$.
\end{cor}

\begin{proof}
First, observe that if $\sigma$ is a normal $p$-chain of $G$ with second term $Q$ then each term of $\sigma$ is contained in $\n_G(Q)$. It follows that, if we define $\sigma_U$ to be the $p$-chain obtained by removing $U$ from $\sigma$, then the assignment $\sigma\mapsto \sigma_U$ defines a bijection between the set of normal $p$-chains of $G$ starting with $U$ and with second term $Q$ and the set of normal $p$-chains of $\n_G(Q)$ starting with $Q$. Moreover, observe that $|\sigma|=|\sigma_U|+1$ and, by assuming as we may that $U=\O_p(G)$, that $\n_A(Q)_{\sigma}=\n_A(U)\cap \n_A(Q)_{\sigma_U}=\n_A(Q)_{\sigma_U}$. Then, we get a bijection
\begin{align*}
\C^d_Q(B,U)_\pm&\to \C^d(B_Q,Q)_\mp
\\
(\sigma,\vartheta)&\mapsto (\sigma_U,\vartheta)
\end{align*}
that preserves the conjugacy action of $\n_A(Q)_B$. Consider now the map $\Omega_{B_Q,Q}$ given by Lemma \ref{lem:CTC for maximal defect follows from iAM, 1} and fix pairs $(\varsigma,\varphi)\in\C^d(B_Q,Q)_+$ and $(\varrho,\psi)\in\Omega_{B_Q,Q}(\overline{(\varsigma,\varphi)})$. Write $(\varsigma,\varphi)=(\rho_U,\chi)$ and $(\varrho,\psi)=(\sigma_U,\vartheta)$ for $(\rho,\chi)\in\C^d_Q(B,U)_-$ and $(\sigma,\vartheta)\in\C^d_Q(B,U)_+$. We then define the map $\Theta_Q$ by sending the $\n_G(Q)$-orbit of the pair $(\sigma,\vartheta)$ to the $\n_G(Q)$-orbit of $(\rho,\chi)$ constructed above. Notice that $\Theta_Q$ is $\n_A(Q)_B$-equivariant since so is $\Omega_{B_Q,Q}$. Moreover, observe that the $\n_G(Q)$-block isomorphism is an equivalence relation which is in particular reflexive. Then, since the character triples $(\n_A(Q)_{\sigma,\vartheta},\n_G(Q)_\sigma,\vartheta)$ and $(\n_A(Q)_{\rho,\chi},\n_G(Q)_\rho,\chi)$ coincide with the character triples $(\n_A(Q)_{\varrho,\psi},\n_G(Q)_\varrho,\psi)$ and $(\n_A(Q)_{\varsigma,\varphi},\n_G(Q)_\varsigma,\varphi)$ respectively, the $\n_G(Q)$-block isomorphism in the statement above coincides with that given by Lemma \ref{lem:CTC for maximal defect follows from iAM, 1}.
\end{proof}

In the next proposition, whose statement will be used in the proof of Theorem \ref{thm:CTC for maximal defect follows from iAM}, we combine the bijections $\Theta_Q$ for all $p$-subgroups $Q$ belonging to a $G$-conjugacy class. Given a $p$-subgroup $Q$ satisfying $U<Q$, we denote by $\overline{Q}$ its $G$-orbit and by $\C^d_{\overline{Q}}(B,U)$ the subset of $\C^d(B,U)$ consisting of those pairs $(\sigma,\vartheta)$ such that the second term of the $p$-chain $\sigma$ is $G$-conjugate to $Q$. Equivalently, $\C^d_{\overline{Q}}(B,U)$ is the set of all the pairs of $\C^d(B,U)$ that are $G$-conjugate to some pair of $\C^d_Q(B,U)$. Notice that $G\n_A(U,Q)_B$-acts on $\C^d_{\overline{Q}}(B,U)$ and denote by $\C^d_{\overline{Q}}(B,U)_\pm$ the intersection of $\C^d_{\overline{Q}}(B,U)$ with $\C^d(B,U)_\pm$.

\begin{prop}
\label{prop:CTC for maximal defect follows from iAM, 3}
Let $Q$ be a $p$-subgroup of $G$ satisfying $U<Q<D$, for some defect group $D$ of $B$, and denote by $\overline{Q}$ its $G$-orbit. Then, for $d:=d(B)$, there exists a $G\n_A(U,Q)_B$-equivariant bijection
\begin{equation}
\label{eq:CTC for maximal defect follows from iAM, 3, 1}
\Theta_{\overline{Q}}:\C^d_{\overline{Q}}(B,U)_+/G\to\C^d_{\overline{Q}}(B,U)_-/G
\end{equation}
such that
\[\left(A_{\sigma,\vartheta},G_\sigma,\vartheta\right)\iso{G}\left(A_{\rho,\chi},G_\rho,\chi\right)\]
for every $(\sigma,\vartheta)\in\C^d_{\overline{Q}}(B,U)_+$ and $(\rho,\chi)\in\Theta_{\overline{Q}}(\overline{(\sigma,\vartheta)})$ where we now denote by $\overline{(\sigma,\vartheta)}$ the $G$-orbit of the pair $(\sigma,\vartheta)$.
\end{prop}

\begin{proof}
Throughout the proof we need to differentiate between $G$-orbits and $\n_G(Q)$-orbits of pairs $(\sigma,\vartheta)$. For this reason, we denote by $\mathcal{O}_G(\sigma,\vartheta)$ and $\mathcal{O}_{\n_G(Q)}(\sigma,\vartheta)$ the $G$-orbit and the $\n_G(Q)$-orbit of $(\sigma,\vartheta)$ respectively. Suppose that $(\sigma,\vartheta)$ belongs to $\C^d_{\overline{Q}}(B,U)_+$ and fix $g\in G$ such that $(\sigma,\vartheta)^g$ belongs to $\C^d_Q(B,U)_+$. If $\Theta_Q$ is the map given by Corollary \ref{cor:CTC for maximal defect follows from iAM, 2}, then choose $(\rho,\chi)$ in $\C^d_Q(B,U)_-$ such that $\mathcal{O}_{\n_G(Q)}(\rho,\chi)=\Theta_Q(\mathcal{O}_{\n_G(Q)}((\sigma,\vartheta)^g))$. We define
\[\Theta_{\overline{Q}}\left(\mathcal{O}_G(\sigma,\vartheta)\right):=\mathcal{O}_G(\rho,\chi)\]
and claim that $\Theta_{\overline{Q}}$ is a well-defined $G\n_A(U,Q)_B$-equivariant bijection between $\C^d_{\overline{Q}}(B,U)_+/G$ and $\C^d_{\overline{Q}}(B,U)_-/G$. First, suppose that $h\in G$ and $(\sigma,\vartheta)^h$ belongs to $\C^d_Q(B,U)_+$. If $D_1$ is the second term of the $p$-chain $\sigma$, then it follows that $Q^{g^{-1}}=D_1=Q^{h^{-1}}$ so that $h^{-1}g\in\n_G(Q)$ and hence $\mathcal{O}_{\n_G(Q)}((\sigma,\vartheta)^h)=\mathcal{O}_{\n_G(Q)}((\sigma,\vartheta)^{hh^{-1}g})=\mathcal{O}_{\n_G(Q)}((\sigma,\vartheta)^g)$. In particular, we get $\Theta_Q(\mathcal{O}_{\n_G(Q)}((\sigma,\vartheta)^h))=\mathcal{O}_{\n_G(Q)}(\rho,\chi)$. This shows that the definition of $\Theta_{\overline{Q}}$ does not depend on the choice of the element $g\in G$ while it is clear that it does not depend on the choice of the representative $(\rho,\chi)$ in the $\n_G(Q)$-orbit $\Theta_Q(\mathcal{O}_{\n_G(Q)}(\sigma,\vartheta)^g)$. It also follows that the map $\Theta_{\overline{Q}}$ is $G$-equivariant. Let now $x\in\n_A(U,Q)_B$. By the above argument, we can assume that the pair $(\sigma,\vartheta)$ belongs to $\C^d_Q(B,U)_+$. Then, since $\Theta_Q$ is $\n_A(U,Q)_B$-equivariant, we get $\Theta_Q(\mathcal{O}_{\n_G(Q)}(\sigma,\vartheta)^x)=\Theta_Q(\mathcal{O}_{\n_G(Q)}(\sigma,\vartheta))^x=\mathcal{O}_{\n_G(Q)}(\rho,\chi)^x$ from which we obtain $\Theta_{\overline{Q}}(\mathcal{O}_G(\sigma,\vartheta)^x)=\mathcal{O}_G(\rho,\chi)^x$. This proves our claim.

Next, we prove the condition on character triples. Keep $(\sigma,\vartheta)$ and $(\rho,\chi)$ as before. Recall that, up to $G$-conjugation, we may assume in the definition of $\Theta_{\overline{Q}}$ that $Q$ coincides with the second term of $\sigma$ and of $\rho$. Moreover, since $G$-block isomorphisms are compatible with $G$-conjugation according to \cite[Lemma 3.8 (c)]{Spa17}, this assumption is compatible with the condition on character triples. Then, since the $\n_G(Q)$-orbits of $(\sigma,\vartheta)$ and $(\rho,\chi)$ correspond under $\Theta_Q$, Corollary \ref{cor:CTC for maximal defect follows from iAM, 2} yields
\begin{equation}
\label{eq:CTC for maximal defect follows from iAM, 3, 2}
\left(\n_A(Q)_{\sigma,\vartheta},\n_G(Q)_\sigma,\vartheta\right)\iso{\n_G(Q)}\left(\n_A(Q)_{\rho,\chi},\n_G(Q)_\rho,\chi\right).
\end{equation}
Furthermore, observe that since $Q$ is a term of the $p$-chains $\sigma$ and $\rho$ we have $A_\sigma=\n_A(Q)_\sigma$ and $A_\rho=\n_A(Q)_\rho$. We can then rewrite \eqref{eq:CTC for maximal defect follows from iAM, 3, 2} as 
\begin{equation}
\label{eq:CTC for maximal defect follows from iAM, 3, 3}
\left(A_{\sigma,\vartheta},G_\sigma,\vartheta\right)\iso{\n_G(Q)}\left(A_{\rho,\chi},G_\rho,\chi\right).
\end{equation}
To conclude we need to show that the $\n_G(Q)$-block isomorphism \eqref{eq:CTC for maximal defect follows from iAM, 3, 3} is actually a $G$-block isomorphism. This is done by applying \cite[Lemma 2.11]{Ros22}. In fact, if $D$ denotes a defect group of the block of $\vartheta$ in $G_\sigma$, then $Q\leq \O_p(G_\sigma)\leq D$ and we get $\c_{GA_{\sigma,\vartheta}}(D)\leq \n_{GA_{\sigma,\vartheta}}(Q)=\n_G(Q)A_{\sigma,\vartheta}$. A similar argument shows that $\c_{GA_{\rho,\chi}}(P)\leq \n_G(Q)A_{\rho,\chi}$ for a defect group $P$ of the block of $\chi$ in $G_\rho$ hence verifying the hypothesis of \cite[Lemma 2.11]{Ros22}. The proof is now complete.
\end{proof}

We now come to the final step of the proof of Theorem \ref{thm:CTC for maximal defect follows from iAM}.

\begin{proof}[Proof of Theorem \ref{thm:CTC for maximal defect follows from iAM}]
Recall that $U$ is a normal $p$-subgroup of $G$ of order $|U|=p^m$ and let $D$ be a defect group of the block $B$. By assumption $m<d=d(B)$ and it follows from \cite[Theorem 4.8]{Nav98} that $U<D$. We claim that every pair $(\sigma,\vartheta)\in\C^d(B,U)$ is $G$-conjugate to a pair whose corresponding $p$-chain has all of its terms contained in $D$. For this notice that, if $b$ is a block of $G_\sigma$ satisfying $b^G=B$, then we can find a defect group $P$ of $b$ and an element $g\in G$ such that $P\leq D^g$ according to \cite[Lemma 4.13]{Nav98}. Now, if $D(\sigma)$ denotes the last term of $\sigma$, then \cite[Theorem 4.8]{Nav98} implies that $D(\sigma)\leq \O_p(G_\sigma)\leq P\leq D^g$. By replacing $(\sigma,\vartheta)$ with $(\sigma,\vartheta)^{g^{-1}}$ we obtain a pair with the properties required above. Thus, we can write $\sigma=\{D_0=U<D_1<\dots<D(\sigma)\}$ with $D(\sigma)\leq D$ and observe that either $|\sigma|=0$, which leads to the $p$-chain $\sigma_+=\{D_0=U\}$, or $|\sigma|\geq 1$ in which case we can have either $U<D_1<D$ or $U<D_1=D$, which leads to the $p$-chain $\sigma_-=\{D_0=U<D_1=D\}$.

Consider the set $\mathcal{F}$ of $p$-subgroups $Q$ of $G$ satisfying $U<Q^g<D$ for some $g\in G$. We denote by $\mathcal{F}/G$ the corresponding set of $G$-orbits and by $\overline{Q}$ the $G$-orbit of $Q$. If $\overline{Q}\in\mathcal{F}/G$ and $x\in\n_A(U)_B$, then $U<Q^g<D$ for some $g\in G$ and $U<Q^{gx}<D^x$. On the other hand, since $x$ stabilises the block $B$, we know that $D^x$ is a defect group of $B$ and so there exists $h\in G$ such that $D^{xh}=D$. It follows that $U<Q^{gxh}<D^{xh}=D$. Furthermore, since $G\unlhd \n_A(U)_B$, we can write $gx=xg'$ for some $g'\in G$ and we conclude that $U<Q^{xg'h}<D$. This shows that $\overline{Q}^x$ belongs to $\mathcal{F}/G$ and therefore the group $\n_A(U)_B$ acts by conjugation on $\mathcal{F}/G$. Fix an $\n_A(U)_B$-transversal $\mathcal{S}$ in $\mathcal{F}/G$ and observe that, for every $\overline{Q}\in\mathcal{S}$, Proposition \ref{prop:CTC for maximal defect follows from iAM, 3} gives a $G\n_A(U,Q)_B$-equivariant bijection
\[\Theta_{\overline{Q}}:\C^d_{\overline{Q}}(B,U)_+/G\to\C^d_{\overline{Q}}(B,U)_-/G\]
such that
\[\left(A_{\sigma,\vartheta},G_\sigma,\vartheta\right)\iso{G}\left(A_{\rho,\chi},G_\rho,\chi\right)\]
for every $(\sigma,\vartheta)\in\C^d_{\overline{Q}}(B,U)_+$ and $(\rho,\chi)\in\Theta_{\overline{Q}}(\overline{(\sigma,\vartheta)})$. In particular, if we fix a $G\n_A(U,Q)_B$-transversal $\mathcal{T}_{\overline{Q}}^+$ in $\C^d_{\overline{Q}}(B,U)_+/G$, then the equivariance properties of $\Theta_{\overline{Q}}$ imply that the image $\mathcal{T}_{\overline{Q}}^-$ of $\mathcal{T}_{\overline{Q}}^+$ under $\Theta_{\overline{Q}}$ is a $G\n_A(U,Q)_B$-transversal in $\C^d_{\overline{Q}}(B,U)_-/G$. If we now define $\C^d_\mathcal{F}(B,U)_\pm$ to be the subset of $\C^d(B,U)_\pm$ consisting of those pairs $(\sigma,\vartheta)$ such that the second term of $\sigma$ belongs to $\mathcal{F}$, then we conclude from the above discussion that
\[\mathcal{T}_\mathcal{F}^+:=\coprod_{\overline{Q}\in\mathcal{S}}\mathcal{T}_{\overline{Q}}^+\]
and
\[\mathcal{T}_\mathcal{F}^-:=\coprod_{\overline{Q}\in\mathcal{S}}\mathcal{T}_{\overline{Q}}^-\]
are $\n_A(U)_B$-transversals in $\C^d_\mathcal{F}(B,U)_+/G$ and $\C^d_\mathcal{F}(B,U)_-/G$ respectively. This follows from the fact that, by a Frattini argument, $G\n_A(U,Q)_B$ coincides with the stabiliser of the $G$-orbit $\overline{Q}$ under the action of $\n_A(U)_B$. We can then define an $\n_A(U)_B$-equivariant bijection
\[\Omega_\mathcal{F}:\C^d_\mathcal{F}(B,U)_+/G\to\C^d_\mathcal{F}(B,U)_-/G\]
by defining
\[\Omega_\mathcal{F}\left(\overline{(\sigma,\vartheta)}^x\right):=\overline{(\rho,\chi)}^x\]
for every $\overline{Q}\in\mathcal{S}$, every $\overline{(\sigma,\vartheta)}\in\mathcal{T}^+_{\overline{Q}}$ corresponding to $\overline{(\rho,\chi)}\in\mathcal{T}_{\overline{Q}}^-$ via $\Theta_{\overline{Q}}$, and every $x\in\n_A(U)_B$. By the properties of the maps $\Theta_{\overline{Q}}$, we get that the map $\Omega_\mathcal{F}$ satisfies the required condition on character triples.

Following the first paragraph of the proof, observe that the set $\C^d(B,U)_\pm$ can be partitioned into the subsets $\C^d_\mathcal{F}(B,U)_\pm$ and $\mathcal{G}_\pm$ where we define $\mathcal{G}_\pm$ as the set of those pairs $(\sigma,\vartheta)$ such that $\sigma$ is $G$-conjugate to $\sigma_\pm$. Notice that $\mathcal{G}_+$ is the set of pairs $(\sigma_+,\vartheta)$ with $\vartheta\in\irr^d(G_{\sigma_+})$ such that $\bl(\vartheta)^G=B$. Equivalently, since $\sigma_+$ is $G$-invariant and $d=d(B)$, the set $\mathcal{G}_+$ consists of those pairs $(\sigma_+,\vartheta)$ where $\vartheta$ is a character of $p$-height zero in the block $B$. In particular, if $\mathcal{S}_\mathcal{G}^+$ is an $\n_A(U,D)_B$-transversal in $\irr_0(B)$, then the set $\mathcal{T}_\mathcal{G}^+$ of $G$-orbits $\overline{(\sigma_+,\vartheta)}$ with $\vartheta\in\mathcal{S}_\mathcal{G}^+$ is a $G\n_A(U,D)_B$-transversal in $\mathcal{G}_+/G$. Next, since by hypothesis the inductive Alperin--McKay condition (as stated in Conjecture \ref{conj:iAM}) holds for every covering group of a non-abelian finite simple group involved in $G$, we can apply \cite[Theorem 7.1]{Nav-Spa14I} with respect to $G\unlhd \n_A(U)_B$ to obtain an $\n_A(U,D)_B$-equivariant bijection
\[\Pi_{B,D}:\irr_0(B)\to\irr_0(C)\]
where $C$ is the Brauer correspondent of $B$ in $\n_G(D)$. Moreover, we have
\begin{equation}
\label{eq:Proof of theorem iAM implies CTC max, 1}
\left(\n_A(U)_{B,\vartheta},G,\vartheta\right)\iso{G}\left(\n_A(U,D)_{B,\chi},\n_G(D),\chi\right)
\end{equation}
for every $\vartheta\in\irr_0(B)$ and $\chi=\Pi_{B,D}(\vartheta)$. Now the image $\mathcal{S}_{\mathcal{G}}^-$ of $\mathcal{S}_{\mathcal{G}}^+$ via the map $\Pi_{B,D}$ is an $\n_A(U,D)_B$-transversal in the set $\irr_0(C)$. Noticing that the set $\mathcal{G}_-$ consists of pairs of the form $(\sigma_-,\chi)^g$ for some $\chi\in\irr_0(C)$ and $g\in G$, we deduce that the set $\mathcal{T}_{\mathcal{G}}^-$ of $G$-orbits $\overline{(\sigma_-,\chi)}$ with $\chi\in\mathcal{S}_\mathcal{G}^-$ is a $G\n_A(U,D)_B$-transversal in $\mathcal{G}_-$. A Frattini argument also shows that $\n_A(U)_B=G\n_A(U,D)_B$. We can now define an $\n_A(U)_B$-equivariant bijection
\[\Omega_\mathcal{G}:\mathcal{G}_+/G\to\mathcal{G}_-/G\]
by defining
\[\Omega_\mathcal{G}\left(\overline{(\sigma_+,\vartheta)}^x\right):=\overline{\left(\sigma_-,\chi\right)}^x\]
for every $\overline{(\sigma_+,\vartheta)}\in\mathcal{T}_\mathcal{G}^+$ corresponding to $\overline{(\sigma_-,\chi)}\in\mathcal{T}_\mathcal{G}^-$ and every $x\in\n_A(U)_B$. The $G$-block isomorphism of character triples \eqref{eq:Proof of theorem iAM implies CTC max, 1} can be rewritten as
\[\left(A_{\sigma_+,\vartheta},G_{\sigma_+},\vartheta\right)\iso{G}\left(A_{\sigma_-,\chi},G_{\sigma_-},\chi\right).\]
We can now construct a map $\Omega$ with the properties required above by defining it to be $\Omega_\mathcal{F}$ and $\Omega_\mathcal{G}$ on the subsets $\C^d_\mathcal{F}(B,U)_+/G$ and $\mathcal{G}_+/G$ respectively. This concludes the proof.
\end{proof}

\section{Proof of Theorem \ref{thm:Main, CTC/DPC-max for p=2} and Corollary \ref{cor:Main CTC for abelian 2-blocks}}
\label{sec:Theorem A ad Corollary B}

We now obtain Theorem \ref{thm:Main, CTC/DPC-max for p=2} as a consequence of Theorem \ref{thm:CTC for maximal defect follows from iAM}.

\begin{proof}[Proof of Theorem \ref{thm:Main, CTC/DPC-max for p=2}]
As mentioned previously, the statement of Theorem \ref{thm:CTC for maximal defect follows from iAM} implies the Character Triple Conjecture in the form introduced in \cite[Conjecture 6.3]{Spa17} if we assume $U\leq \z(G)$. Furthermore, in this case it is no loss of generality to assume that $U=\O_p(G)$ according to \cite[Lemma 2.3]{Ros22}. In order to apply Theorem \ref{thm:CTC for maximal defect follows from iAM} observe that the inductive Alperin--McKay condition has been verified for the prime $p=2$ with respect to alternating simple groups \cite{Den14}, \cite[Corollary 8.3 (a)]{Spa13II}, Suzuki and Ree groups \cite[Theorem 1.1]{Mal14}, sporadic groups \cite{Bre-iAM}, groups of Lie type with exceptional Schur multiplier \cite{Bre-iAM}, \cite[Lemma 7.3]{Bro-Ruh22}, groups of Lie type in characteristic $2$ \cite[Proposition 14.8]{Ruh21}, classical groups of Lie type in odd characteristic \cite[Corollary 8.1]{Bro-Ruh22}, and finally exceptional groups of Lie type in odd characteristic \cite[Theorem C]{Ruh22AM}, \cite[Theorem C]{Ruh22}.
\end{proof}

As claimed in the introduction, using Theorem \ref{thm:Main, CTC/DPC-max for p=2} and Brauer's Height Zero Conjecture, we can prove that the Character Triple Conjecture holds for every $2$-block with abelian defect groups.

\begin{proof}[Proof of Corollary \ref{cor:Main CTC for abelian 2-blocks}]
Let $G\unlhd A$ be finite groups and $U\unlhd G$ a $2$-subgroup. Consider a $2$-block $B$ of $G$ with abelian defect group $D$ such that $U<D$ and write $d:=d(B)$. By Theorem \ref{thm:Main, CTC/DPC-max for p=2} there exists a bijection $\C^d(B,U)_+/G\to\C^d(B,U)_-/G$ as required by the Character Triple Conjecture. So it remains to show that such a bijection can be constructed by replacing $d$ with any other non-negative integer, say $0\leq f\neq d$. For this consider a pair $(\sigma,\vartheta)\in\C^f(B,U)_\pm$ so that $\vartheta$ is an irreducible character of the stabiliser $G_\sigma$ of defect $d(\vartheta)=f$ and whose block satisfies $\bl(\vartheta)^G=B$. Observe that $\sigma\neq \sigma_+:=\{D_0=U\}$ and $\sigma\neq \sigma_-:=\{D_0=U<D_1=D\}$. In fact, $G_{\sigma_+}=G$, $G_{\sigma_-}=\n_G(D)$ and, since $D$ is abelian, Brauer's Height Zero Conjecture (we actually only need the half proved in \cite{Kes-Mal13}) implies that $\irr(B)=\irr^d(B)$ and that $\irr(b)=\irr^d(b)$ where $b$ is the Brauer correspondent of $B$ in $\n_G(D)$. In particular the $2$-chain $\sigma$ belongs to the set $\mathcal{F}$ defined in the final step of the proof of Theorem \ref{thm:CTC for maximal defect follows from iAM}. Now proceeding by induction on the order of $G$ and arguing as in Lemma \ref{lem:CTC for maximal defect follows from iAM, 1}, Corollary \ref{cor:CTC for maximal defect follows from iAM, 2} and Proposition \ref{prop:CTC for maximal defect follows from iAM, 3}, it suffices to exhibit an $\n_A(Q)_b$-equivariant bijection $\C^f(b,Q)_+/\n_G(Q)\to \C^f(b,Q)_-/\n_G(Q)$ inducing $\n_G(Q)$-block isomorphisms for every $U<Q<D$ and every block $b$ of $\n_G(Q)$ satisfying $b^G=B$. In other words, we need to show that the Character Triple Conjecture holds for the $2$-block $b$ of $\n_G(Q)$ with respect to $f$. This follows by induction since the condition $b^G=B$ implies that $b$ has abelian defect groups according to \cite[Lemma 4.13]{Nav98}.
\end{proof}

\section{A converse to Theorem \ref{thm:Main, reduction for CTC-max}}
\label{sec:Converse, CTC implies iAM}

It was shown by Dade in \cite{Dad94} that the projective form of his conjecture implies the Alperin--McKay Conjecture. Later, Navarro \cite[Theorem 9.27]{Nav18} proved that the block-free version of Dade's Ordinary Conjecture implies the McKay Conjecture, while Kessar and Linckelmann \cite{Kes-Lin19} extended these results by proving that Dade's Ordinary Conjecture implies the Alperin--McKay Conjecture. It is therefore natural to ask whether the Character Triple Conjecture, which plays the role of an inductive condition for Dade's Projective Conjecture, implies the inductive Alperin--McKay condition. In this section, we show that this is the case and obtain the following result which can be seen as a converse to Theorem \ref{thm:Main, reduction for CTC-max}.

\begin{theo}
\label{thm:CTC-max implies iAM}
If the Character Triple Conjecture holds for maximal defect characters at the prime $p$, then the inductive Alperin--McKay condition (in the generality considered in Conjecture \ref{conj:iAM}) holds at the prime $p$.
\end{theo}

The structure of a minimal counterexample $G$ to Conjecture \ref{conj:iAM} has been studied in \cite[Section 7]{Nav-Spa14I}. In particular, according to \cite[Proposition 7.4]{Nav-Spa14I} we know that $\O_p(G)$ is contained in the centre of $G$.

\begin{lem}
\label{lem:Minimal counterexample to iAM, central O_p}
Let $G\unlhd A$ be a minimal counterexample to Conjecture \ref{conj:iAM} with respect to $|G:\z(G)|$. Then $\O_p(G)\leq\z(G)$.
\end{lem}

Now, let $G\unlhd A$ be a minimal counterexample as in Lemma \ref{lem:Minimal counterexample to iAM, central O_p} and consider a block $B$ of $G$ for which Conjecture \ref{conj:iAM} fails to hold. If $D$ is a defect group of $B$, then $\O_p(G)<D$ for if $\O_p(G)=D$ then $D$ is normal in $G$ and Conjecture \ref{conj:iAM} follows trivially. Then, for every non-negative integer $d$ we can define the sets $\C^d_0(B,\O_p(G))$ and $\C^d_1(B,\O_p(G))$ consisting of those pairs $(\sigma,\vartheta)$ belonging to $\C^d(B,\O_p(G))$ and such that $\sigma=\{D_0=\O_p(G)\}$ and $\sigma=\{D_0=\O_p(G)<D_1\}$, with $D_1$ a defect group of $B$, respectively. Moreover, set $\mathcal{J}_+^d:=\C^d(B,\O_p(G))_+\setminus \C_0^d(B,\O_p(G))$ and $\mathcal{J}_-^d:=\C^d(B,\O_p(G))_-\setminus \C^d_1(B,\O_p(G))$. Notice that $G$ acts by conjugation on $\mathcal{J}^d_\pm$ and let $\mathcal{J}^d_\pm/G$ denote the corresponding set of $G$-orbits. As usual, for any element $(\sigma,\vartheta)\in\mathcal{J}_\pm^d$, we denote its $G$-orbit by $\overline{(\sigma,\vartheta)}$.

\begin{prop}
\label{prop:Bijection Pi}
Let $G\unlhd A$ be finite groups with $G$ a minimal counterexample to Conjecture \ref{conj:iAM} with respect to $|G:\z(G)|$ and consider a block $B$ of $G$, with defect group $D$, for which the result fails to hold. If $d:=d(B)$, then there exists an $\n_A(D)_B$-equivariant bijection
\[\Pi:\mathcal{J}_+^d/G\to\mathcal{J}_-^d/G\]
such that
\[\left(A_{\sigma,\vartheta},G_\sigma,\vartheta\right)\iso{G}\left(A_{\rho,\chi},G_\rho,\chi\right),\]
for every $(\sigma,\vartheta)\in\mathcal{J}_+^d$ and $(\rho,\chi)\in\Pi(\overline{(\sigma,\vartheta)})$.
\end{prop}

\begin{proof}
Define the set $\wh{\mathcal{J}}_\pm^d$ of $p$-chains $\sigma$ of $G$ that start with $\O_p(G)$ and for which there exists a character $\vartheta\in\irr(G_\sigma)$ such that $(\sigma,\vartheta)\in\mathcal{J}_\pm^d$. Denote by $\wh{\mathcal{J}}_\pm^d/G$ the corresponding set of $G$-orbits and by $\overline{\sigma}$ the $G$-orbit of $\sigma\in\wh{\mathcal{J}}_\pm^d$. Notice that, if $\sigma\in\wh{\mathcal{J}}_\pm^d$ has final term $D(\sigma)$, then there exists $g\in G$ such that
\[D(\sigma)\leq D^g\leq G_\sigma\]
and $D^g$ is a defect group of some block of $G_\sigma$. In fact, if $(\sigma,\vartheta)\in\mathcal{J}_\pm^d$ and $Q$ is a defect group of $\bl(\vartheta)$, then $D(\sigma)\leq \O_p(G_\sigma)\leq Q$ according to \cite[Theorem 4.8]{Nav98} while \cite[Lemma 4.13]{Nav98} implies that there exists $g\in G$ such that $Q\leq D^g$. Furthermore, if $f$ denotes the defect of the block $\bl(\vartheta)$, then $d\leq f$ by \cite[Theorem 4.6]{Nav98} and hence we have $d\leq f\leq d(\bl(\vartheta)^G)=d(B)=:d$. This shows that $D^g=Q\leq G_\sigma$ and thus $D(\sigma)\leq D^g\leq G_\sigma$ as claimed.

Next, we define an $\n_A(D)_B$-equivariant bijection
\[\wh{\Pi}:\wh{\mathcal{J}}_+^d/G\to\wh{\mathcal{J}}_-^d/G\]
by sending the $G$-orbit of the $p$-chain $\sigma$ to the $G$-orbit of the $p$-chain $\rho$ obtained by deleting the final term $D(\sigma)$ if $D(\sigma)$ is a defect group of $B$. If $D(\sigma)$ is not a defect group of $B$, then the above discussion implies that there exists $g\in G$ such that $D(\sigma)<D^g$ and $D^g$ is a defect group of a block of the stabiliser $G_\sigma$. In this case, we define $\wh{\Pi}$ by sending the $G$-orbit of $\sigma$ to the $G$-orbit of the $p$-chain $\rho$ obtained by adding the term $D^g$ at the end of the $p$-chain $\sigma$. Notice that the above definition does not depend on the choice of $D^g$, but only on its $G_\sigma$-conjugacy class, nor on the representative $\sigma$ in $\overline{\sigma}$. Furthermore, as $D^g\leq G_\sigma$, we deduce that the map sends normal $p$-chains to normal $p$-chains. To conclude that $\wh{\Pi}$ is well-defined we need to check that, for every $\rho\in\wh{\Pi}(\overline{\sigma})$, there exists $\chi\in G_\rho$ such that $(\rho,\chi)\in\mathcal{J}_-^d$. Without loss of generality we may assume that $\rho$ is the $p$-chain obtain from $\sigma$ by adding $D$ as a final term and that, if $(\sigma,\vartheta)\in\mathcal{J}^d_+$, the block $b:=\bl(\vartheta)$ has defect group $D$. Notice also that by the definition of the sets $\mathcal{J}_\pm^d$, since we are excluding the $p$-chain $\{D_0=\O_p(G)\}$, we get $G_\sigma<G$ because the last term of $\sigma$ properly contains $\O_p(G)$. In particular, we deduce that $|G_\sigma:\z(G_\sigma)|<|G:\z(G)|$ and thus $G_\sigma$ satisfies Conjecture \ref{conj:iAM} by the minimality of $G$. Then, if $c$ is the Brauer correspondent of $\bl(\vartheta)$ in $\n_{G_\sigma}(D)=G_\rho$, then there exists an $A_\rho$-equivariant bijection
\[\Pi_{\sigma,b}:\irr^d(b)\to\irr^d(c)\]
such that
\[\left(A_{\sigma,\vartheta},G_\sigma,\vartheta\right)\iso{G_\sigma}\left(A_{\rho,\vartheta},G_\rho,\Pi_\sigma(\vartheta)\right),\]
for every $\vartheta\in\irr^d(b)$. Noticing that $\c_{A_{\sigma,\vartheta}\cdot G}(D)\leq A_{\sigma,\vartheta}$ and applying \cite[Lemma 2.11]{Ros22} we can use the above $G_\sigma$-block isomorphism of character triples to get 
\[\left(A_{\sigma,\vartheta},G_\sigma,\vartheta\right)\iso{G}\left(A_{\rho,\vartheta},G_\rho,\Pi_\sigma(\vartheta)\right).\]
In particular, for $\chi:=\Pi_{\sigma,b}(\vartheta)$, we have $(\rho,\chi)\in\mathcal{J}_-^d$ and so $\wh{\Pi}$ is well-defined as explained above. Finally, we use the bijections $\wh{\Pi}$ and $\Pi_{\sigma,b}$ to define an $\n_A(D)_B$-equivariant bijection $\Pi:\mathcal{J}_+^d/G\to\mathcal{J}_-^d/G$ as required in the statement by sending the $G$-orbit of $(\sigma,\vartheta)$ to the $G$-orbit of the pair $(\rho,\chi)$ constructed above. 
\end{proof}

We can now prove Theorem \ref{thm:CTC-max implies iAM}.

\begin{proof}[Proof of Theorem \ref{thm:CTC-max implies iAM}]
Let $G\unlhd A$ be finite groups and assume that $G$ is a minimal counterexample to Conjecture \ref{conj:iAM} with respect to $|G:\z(G)|$. Let $B$ be a block of $G$ with defect group $D$ and Brauer correspondent $b$ in $\n_G(D)$ for which the result fails to holds. 
By Lemma \ref{lem:Minimal counterexample to iAM, central O_p} and the discussion preceding it we know that $\O_p(G)\leq \z(G)$ and that $\O_p(G)<D$. Then, since we are assuming that the Character Triple Conjecture holds for the non-negative integer $d=d(B)$ at the prime $p$, we can find an $A_B$-equivariant bijection
\[\Omega:\C^d(B,\O_p(G))_+/G\to \C^d(B,\O_p(G))_-/G,\]
such that
\[\left(A_\vartheta,G,\vartheta\right)\iso{G}\left(\n_A(D)_\vartheta,\n_G(D),\Omega(\vartheta)\right),\]
for every $(\sigma,\vartheta)\in\C^d(B,\O_p(G))_+$ and $(\rho,\chi)\in\Omega(\overline{(\sigma,\vartheta)})$. Consider the sets $\C^d_0(B,\O_p(G))$ and $\C^d_1(B,\O_p(G))$ defined before Proposition \ref{prop:Bijection Pi} and notice that $\C^d_0(B,\O_p(G))/G$ is the set of $G$-orbits of pairs $(\sigma,\vartheta)$ where $\sigma=\{D_0=\O_p(G)\}$ and $\vartheta\in\irr_0(B)$ while $\C^d_1(B,\O_p(G))/G$ is the set of $G$-orbits of pairs $(\rho,\chi)$ with $\rho=\{D_0=\O_p(G)<D_1=D\}$ and $\chi\in\irr_0(b)$. Suppose that the bijection $\Omega$ maps $\C_0^d(B,\O_p(G))/G$ onto $\C^d_1(B,\O_p(G))/G$. If the $G$-orbit of $(\sigma,\vartheta)$ is mapped to that of $(\rho,\chi)$, then we write $\chi:=\Theta(\vartheta)$ and obtain an $\n_A(D)_B$-equivariant bijection 
\[\Theta:\irr_0(B)\to \irr_0(b)\]
such that
\[\left(A_\vartheta,G,\vartheta\right)\iso{G}\left(\n_A(D)_\vartheta,\n_G(D),\Theta(\vartheta)\right),\]
for every $\vartheta\in\irr_0(B)$ as required by Conjecture \ref{conj:iAM}. This contradicts our choice of $G$ and thus the image of $\C_0^d(B,\O_p(G))/G$ under $\Omega$ cannot coincide with $\C^d_1(B,\O_p(G))/G$. However, if $\Pi$ is the bijection given by Proposition \ref{prop:Bijection Pi}, then we get
\begin{align*}
|\C^d_0(B,\O_p(G))/G|&=|\C^d(B,\O_p(G))_+/G|-|\mathcal{J}_+^d/G|
\\
&=|\Omega(\C^d(B,\O_p(G))_+/G)|-|\Pi(\mathcal{J}_+^d/G)|
\\
&=|\C^d(B,\O_p(G))_-/G|-|\mathcal{J}_-^d/G|
\\
&=|\C^d_1(B,\O_p(G))/G|
\end{align*}
and therefore there exists some element $(\sigma_0,\vartheta_0)$ of $\C^d_0(B,\O_p(G))$ whose $G$-orbit is mapped via $\Omega$ outside the set $\C^d_1(B,\O_p(G))/G$. Now, we proceed as follows: noticing that $\Omega(\overline{(\sigma_0,\vartheta_0)})$ belongs to $\mathcal{J}_-^d/G$, we can apply the inverse of the bijection $\Pi$ and define
\[\overline{(\sigma_1,\vartheta_1)}:=\Pi^{-1}\left(\Omega\left(\overline{(\sigma_0,\vartheta_0)}\right)\right)\]
an element of $\mathcal{J}_+^d/G\subseteq\C^d(B,\O_p(G))_+/G$. We can apply $\Omega$ to $\overline{(\sigma_1,\vartheta_1)}$. If $\Omega(\overline{(\sigma_1,\vartheta_1)})$ belongs to $\C^d_1(B,\O_p(G))/G$, then we stop. Otherwise, as before, the element $\Omega(\overline{(\sigma_1,\vartheta_1)})$ belongs to $\mathcal{J}_-^d/G$ and we define
\[\overline{(\sigma_2,\vartheta_2)}:=\Pi^{-1}\left(\Omega\left(\overline{(\sigma_1,\vartheta_1)}\right)\right).\]
Proceeding this way, for $i\geq 1$, we define a sequence of elements of $\C^d(B,\O_p(G))_+/G$ by setting
\[\overline{(\sigma_i,\vartheta_i)}:=\Pi^{-1}\left(\Omega\left(\overline{(\sigma_{i-1},\vartheta_{i-1})}\right)\right),\]
if $\Omega(\overline{(\sigma_{i-1},\vartheta_{i-1})})\nin\C^d_1(B,\O_p(G))/G$. It is important to observe that, for every $i\geq 1$, the pair $(\sigma_i,\vartheta_i)$ does not belong to $\C^d_0(B,\O_p(G))$ and satisfies the condition
\begin{equation}
\label{eq:thm CTC implies iAM 1}
\left(A_{\sigma_0},G_{\sigma_0},\vartheta_0\right)\iso{G}\left(A_{\sigma_i},G_{\sigma_i},\vartheta_i\right).
\end{equation}
Next, we claim that there exists some integer $n\geq 1$ such that $\Omega(\overline{(\sigma_n,\vartheta_n)})\in\C^d_1(B,\O_p(G))/G$. Assume for the sake of contradiction that this is not the case. Then the set 
\[\mathcal{S}:=\left\lbrace (\Pi^{-1}\circ\Omega)^i\left(\overline{(\sigma_0,\vartheta_0)}\right)\enspace\middle|\enspace i\geq 0\right\rbrace\subseteq \C^d(B,\O_p(G))_+/G\]
is well-defined and its image under $\Omega$ is contained in $\mathcal{J}_-^d/G$. If we apply $\Pi^{-1}$ to $\Omega(\mathcal{S})$, then we obtain a subset of $\mathcal{S}$. Equivalently, the map $\Pi^{-1}\circ\Omega$ maps $\mathcal{S}$ to itself. Therefore, since $\mathcal{S}$ is finite, we must have
\[\Pi^{-1}\circ \Omega(\mathcal{S})=\mathcal{S}.\]
However, noticing that $\overline{(\sigma_0,\vartheta_0)}\in\mathcal{S}\cap \C^d_0(B,\O_p(G))/G$ and recalling that from elementary set theory the image of the intersection of two sets under an injective map coincides with the intersection of the images of such sets, we deduce that
\begin{align*}
|\mathcal{S}|&=|\Omega(\mathcal{S})|
\\
&=|\Omega(\mathcal{S})\cap \mathcal{J}_-^d/G|
\\
&=|\Pi^{-1}\left(\Omega(\mathcal{S})\cap \mathcal{J}_-^d/G\right)|
\\
&=|\Pi^{-1}(\Omega(\mathcal{S}))\cap \Pi(\mathcal{J}_-^d/G)|
\\
&=|\mathcal{S}\cap \mathcal{J}_+^d/G|
\\
&\leq |\mathcal{S}|-1,
\end{align*}
a contradiction. This proves our claim. Now, since $\C^d_1(B,\O_p(G))$ is $\n_A(D)_B$-stable and $\Omega$ and $\Pi$ are $\n_A(D)_B$-equivariant, the pairs $(\sigma_0,\vartheta_0)$ and $(\sigma_n,\vartheta_n)$ are not $\n_A(D)_B$-conjugate. Then, we can find an $\n_A(D)_B$-transversal $\mathcal{T}$ in $\C^d(B,\O_p(G))_+/G$ containing $\overline{(\sigma_0,\vartheta_0)}$ and $\overline{(\sigma_n,\vartheta_n)}$. We define a new $\n_A(D)_B$-equivariant bijection $\Omega':\C^d(B,\O_p(G))_+/G\to\C^d(B,\O_p(G))_-/G$ by setting
\[\Omega'\left(\overline{(\sigma,\vartheta)}^x\right):=\begin{cases}
\Omega\left(\overline{(\sigma,\vartheta)}^x\right), &
\text{ if }\overline{(\sigma,\vartheta)}\in \mathcal{T}\setminus \left\lbrace \overline{(\sigma_0,\vartheta_0)},\overline{(\sigma_n,\vartheta_n)}\right\rbrace
\\
\Omega\left(\overline{(\sigma_n,\vartheta_n)}^x\right), &
\text{ if }\overline{(\sigma,\vartheta)}=\overline{(\sigma_0,\vartheta_0)}
\\
\Omega\left(\overline{(\sigma_0,\vartheta_0)}^x\right), &
\text{ if }\overline{(\sigma,\vartheta)}=\overline{(\sigma_n,\vartheta_n)}
\end{cases},\]
for every $\overline{(\sigma,\vartheta)}\in\mathcal{T}$ and $x\in\n_A(D)_B$. Using \eqref{eq:thm CTC implies iAM 1} and because the $G$-block isomorphism is an equivalence relation according to \cite[Lemma 3.8 (a)]{Spa17}, we deduce that $\Omega'$ satisfies the requirements of the Character Triple Conjecture. Moreover, since by construction $(\sigma_n,\vartheta_n)\nin\C^d_0(B)$, the definition of $\Omega'$ coincides with that of $\Omega$ on the set $\C^d_0(B,\O_p(G))/G$ apart from the value of our "bad" element $(\sigma_0,\vartheta_0)$ which is now mapped to $\C^d_1(B,\O_p(G))/G$ under $\Omega'$. Arguing in this way we can redefine the map $\Omega$ for all such bad elements in such a way that the newly defined $\Omega$ maps $\C^d_0(B,\O_p(G))/G$ to $\C^d_1(B,\O_p(G))$. As explained at the beginning of the proof this implies that Conjecture \ref{conj:iAM} holds for $B$. This contradicts our choice of a minimal counterexample and the proof is now complete.
\end{proof}

\section{The block-free form of the Character Triple Conjecture}
\label{sec:Block-free}

In this section, we consider a block-free analogue of Theorem \ref{thm:CTC for maximal defect follows from iAM}. For this purpose, given a non-negative integer $d$ and a normal $p$-subgroup $U$ of $G$, we denote by $\C^d(G,U)_\pm$ the union of all sets $\C^d(B,U)_\pm$ for $B$ a block of $G$. Equivalently, $\C^d(G,U)_\pm$ is the set of pairs $(\sigma,\vartheta)$ where $\sigma$ is a $p$-chain of $G$ starting with $U$ and satisfying $(-1)^{|\sigma|}=\pm 1$, and $\vartheta$ is an irreducible character of the stabiliser $G_\sigma$ with defect $d(\vartheta)=d$. Moreover, observe that by removing the condition \cite[Remark 3.7 (iv)]{Spa17} from the definition of $G$-block isomorphism, we obtain a weaker isomorphism of character triples. This was called $G$-central isomorphism in \cite[Definition 3.3.4]{Ros-Thesis}. With these definitions at hand, a block-free form of the Character Triple Conjecture was introduced in \cite[Conjecture 3.5.5]{Ros-Thesis}. The case of maximal defect characters, which in this context coincide with characters of $p'$-degree, can then be deduced by assuming the inductive McKay condition form \cite{Isa-Mal-Nav07}. Below we use a reformulation of this condition in the spirit of Conjecture \ref{conj:iAM}. We refer the reader to \cite[Conjecture A]{Ros-iMcK} for a precise statement.

\begin{theo}
\label{thm:CTC non-blockwise for maximal defect follows from iMcK}
Let $G$ be a finite group, consider a prime $p$, and suppose that the inductive McKay condition (as stated in \cite[Conjecture A]{Ros-iMcK}) holds at the prime $p$ for the universal covering group of every non-abelian finite simple group involved in $G$. Let $G\unlhd A$ and $U\unlhd G$ a $p$-subgroup of order $|U|<|G|_p=p^d$. Then, there exists an $\n_A(U)$-equivariant bijection
\[\Omega_{U}:\C^d(G,U)_+/G\to \C^d(G,U)_-/G\]
such that
\[\left(A_{\sigma,\vartheta},G_\sigma,\vartheta\right)\isoc{G}\left(A_{\rho,\chi},G_\rho,\chi\right)\]
for every $(\sigma,\vartheta)\in\C^d(G,U)_+$ and $(\rho,\chi)\in\Omega_U(\overline{(\sigma,\vartheta)})$.
\end{theo}

\begin{proof}
By replacing the defect group $D$ of $B$ with a Sylow $p$-subgroup $P$ of $G$ in the arguments used to prove Lemma \ref{lem:CTC for maximal defect follows from iAM, 1}, Corollary \ref{cor:CTC for maximal defect follows from iAM, 2}, and Proposition \ref{prop:CTC for maximal defect follows from iAM, 3}, we obtain $G\n_A(U,Q)$-equivariant bijections, for $U<Q<P$ a $p$-subgroup,
\[\Theta_{\overline{Q}}:\C^d_{\overline{Q}}(G,U)_+/G\to \C^d_{\overline{Q}}(G,U)_-/G\]
such that
\[\left(A_{\sigma,\vartheta},G_\sigma,\vartheta\right)\isoc{G}\left(A_{\rho,\chi},G_\rho,\chi \right)\]
for every $(\sigma,\vartheta)\in\C^d_{\overline{Q}}(G,U)_+$ and $(\rho,\chi)\in\Theta_{\overline{Q}}(\overline{(\sigma,\vartheta)})$. Then, as in the final step of the proof of Theorem \ref{thm:CTC for maximal defect follows from iAM}, we can combine the bijections $\Theta_{\overline{Q}}$ to obtain an $\n_A(U)$-equivariant bijection
\[\Omega_\mathcal{F}:\C^d_{\mathcal{F}}(G,U)_+/G\to\C^d_{\mathcal{F}}(G,U)_-/G\]
where $\mathcal{F}$ denotes the set of $p$-subgroups $Q$ of $G$ such that $U<Q^g<P$ for some $g\in G$, while $\C^d_\mathcal{F}(G,U)_\pm$ is the set of pairs $(\sigma,\vartheta)\in\C^d(G,U)_\pm$ such that the second term of the $p$-chain $\sigma$ belongs to $\mathcal{F}$. To conclude we define $\mathcal{G}_+$ as the set of pairs $(\sigma_+,\vartheta)$ with $\sigma_+=\{U\}$ and $\vartheta\in\irr_{p'}(G)$, and the set $\mathcal{G}_-$ of pairs $(\sigma_-,\chi)^g$ with $\sigma_-=\{U<P\}$, $\chi\in\irr_{p'}(\n_G(P))$ and $g\in G$. To construct a bijection $\Omega_\mathcal{G}$ that induces $G$-central isomorphisms of character triples between the sets $\mathcal{G}_+/G$ and $\mathcal{G}_-/G$, we now use the hypothesis that the inductive McKay condition holds for the universal covering group of every non-abelian simple group involved in $G$ and apply \cite[Theorem B]{Ros-iMcK}. The bijection $\Omega$ is then constructed using $\Omega_\mathcal{F}$ and $\Omega_\mathcal{G}$. 
\end{proof}

Before proceeding further, we make a remark on the block-free form of Theorem \ref{thm:CTC-max implies iAM}. 

\begin{rmk}
By following the argument used in Section \ref{sec:Converse, CTC implies iAM}, while replacing everywhere the defect group $D$ with a Sylow $p$-subgroup $P$ and Lemma \ref{lem:Minimal counterexample to iAM, central O_p} with \cite[Corollary 4.3]{Ros-iMcK}, one could prove a block-free version of Theorem \ref{thm:CTC-max implies iAM} and hence obtain a converse to Theorem \ref{thm:CTC non-blockwise for maximal defect follows from iMcK}. More precisely, if the block-free version of the Character Triple Conjecture holds for every finite group $G$ at the prime $p$ with respect to $d:=\log_p(|G|_p)$, then the inductive McKay condition (in the formulation given in \cite[Conjecture A]{Ros-iMcK}) holds for every finite group at the prime $p$.
\end{rmk}

While we can obtain the block-free form of the Character Triple Conjecture for the prime $2$ and maximal defect characters as a consequence of Theorem \ref{thm:Main, CTC/DPC-max for p=2}, the above result can be used to handle the remaining odd primes $p$ thanks to the solution of McKay Conjecture recently obtained by Cabanes and Sp\"ath \cite{Cab-Spa-McK}.

\begin{theo}
\label{thm:CTC block-free p=3}
Let $G$ be a finite group and write $|G|_p=p^d$. Then, the block-free form of the Character Triple Conjecture \cite[Conjecture 3.5.5]{Ros-Thesis} holds for $G$ at the prime $p$ and with respect to the defect $d$.
\end{theo}

\begin{proof}
By Theorem \ref{thm:CTC non-blockwise for maximal defect follows from iMcK} it suffices to verify the inductive McKay condition for finite simple groups with respect to the prime $p$. This has been verified for Suzuki and Ree groups \cite[Section 16-17]{Isa-Mal-Nav07}, alternating groups \cite[Theorem 3.1]{Mal08II}, groups of Lie type with exceptional Schur multiplier \cite[Theorem 4.1]{Mal08II}, sporadic groups \cite{Mal08II}, groups of Lie type in defining characteristic \cite[Theorem 1.1]{Spa12}, and groups of Lie type in non-defining characteristic unless of type ${\bf D}$ \cite{Cab-Spa13}, \cite{Cab-Spa17I}, \cite{Cab-Spa17II}, \cite{Cab-Spa19}. The remaining case of groups of Lie type ${\bf D}$ in non-defining characteristic follows from \cite[Theorem 3.1]{Mal-Spa16} and \cite{Cab-Spa-McK} by applying \cite[Theorem 2.4]{Cab-Spa19}.
\end{proof}

As a consequence of the above theorem we obtain the block-free form of the Character Triple Conjecture for every prime $p$ and every finite group $G$ with abelian Sylow $p$-subgroup. 

\begin{cor}
\label{cor:CTC block-free, abelian Sylow 3}
The block-free version of the Character Triple Conjecture holds at the prime $p$ for every finite group with abelian Sylow $p$-subgroups.
\end{cor}

\begin{proof}
Let $G$ be a finite group with abelian Sylow $p$-subgroups and $U$ a normal subgroup of $G$ satisfying $|U|<|G|_p=:p^d$. Then, applying Theorem \ref{thm:CTC block-free p=3}, we get a bijection $\Omega_U^d:\C^d(G,U)_+/G\to \C^d(G,U)_-/G$ satisfying the requirements of the block-free form of the Character Triple Conjecture. On the other hand, by the Ito--Michler theorem, we deduce that $\C^f(G,U)_\pm$ is empty whenever $f\neq d$ and we are done.  
\end{proof}

\bibliographystyle{alpha}
\bibliography{References}

\begin{thebibliography}{MNSFT22}

\bibitem[Bre]{Bre-iAM}
T.~Breuer.
\newblock Computations for some simple groups.
\newblock \href{http://www.math.rwth-aachen.de/~Thomas.Breuer/ctblocks/doc/overview.html}{http://www.math.rwth-aachen.de/$\sim$Thomas.Breuer/ctblocks/doc/overview.html}.

\bibitem[BR22]{Bro-Ruh22}
J.~Brough and L.~Ruhstorfer.
\newblock On the {A}lperin--{M}c{K}ay conjecture for 2-blocks of maximal defect.
\newblock {\em J. Lond. Math. Soc. (2)}, 106(2):1580--1602, 2022.

\bibitem[CS]{Cab-Spa-McK}
M.~Cabanes and B.~Sp{\"a}th.
\newblock The {M}c{K}ay {C}onjecture on character degrees.
\newblock {\em in preparation}.

\bibitem[CS13]{Cab-Spa13}
M.~Cabanes and B.~Sp{\"a}th.
\newblock Equivariance and extendibility in finite reductive groups with
  connected center.
\newblock {\em Math. Z.}, 275(3-4):689--713, 2013.

\bibitem[CS17a]{Cab-Spa17I}
M.~Cabanes and B.~Sp{\"a}th.
\newblock Equivariant character correspondences and inductive {M}c{K}ay
  condition for type {$\bf A$}.
\newblock {\em J. Reine Angew. Math.}, 728:153--194, 2017.

\bibitem[CS17b]{Cab-Spa17II}
M.~Cabanes and B.~Sp{\"a}th.
\newblock Inductive {M}c{K}ay condition for finite simple groups of type
  {$\bf{C}$}.
\newblock {\em Represent. Theory}, 21:61--81, 2017.

\bibitem[CS19]{Cab-Spa19}
M.~Cabanes and B.~Sp{\"a}th.
\newblock Descent equalities and the inductive {M}c{K}ay condition for types
  {$\bf B$} and {$\bf E$}.
\newblock {\em Adv. Math.}, 356:106820, 48, 2019.

\bibitem[Dad92]{Dad92}
E.~C. Dade.
\newblock Counting characters in blocks. {I}.
\newblock {\em Invent. Math.}, 109(1):187--210, 1992.

\bibitem[Dad94]{Dad94}
E.~C. Dade.
\newblock Counting characters in blocks. {II}.
\newblock {\em J. Reine Angew. Math.}, 448:97--190, 1994.

\bibitem[Dad97]{Dad97}
E.~C. Dade.
\newblock Counting characters in blocks. {II}.9.
\newblock In {\em Representation theory of finite groups ({C}olumbus, {OH},
  1995)}, volume~6 of {\em Ohio State Univ. Math. Res. Inst. Publ.}, pages
  45--59. de Gruyter, Berlin, 1997.

\bibitem[Den14]{Den14}
D.~Denoncin.
\newblock Inductive {AM} condition for the alternating groups in characteristic
  2.
\newblock {\em J. Algebra}, 404:1--17, 2014.

\bibitem[IMN07]{Isa-Mal-Nav07}
I.~M. Isaacs, G.~Malle, and G.~Navarro.
\newblock A reduction theorem for the {M}c{K}ay conjecture.
\newblock {\em Invent. Math.}, 170(1):33--101, 2007.

\bibitem[KL19]{Kes-Lin19}
R.~Kessar and M.~Linckelmann.
\newblock Dade's ordinary conjecture implies the {A}lperin--{M}c{K}ay
  conjecture.
\newblock {\em Arch. Math. (Basel)}, 112(1):19--25, 2019.

\bibitem[KM13]{Kes-Mal13}
R.~Kessar and G.~Malle.
\newblock Quasi-isolated blocks and {B}rauer's height zero conjecture.
\newblock {\em Ann. of Math. (2)}, 178(1):321--384, 2013.

\bibitem[KR89]{Kno-Rob89}
R.~Kn{\"o}rr and G.~R. Robinson.
\newblock Some remarks on a conjecture of {A}lperin.
\newblock {\em J. London Math. Soc. (2)}, 39(1):48--60, 1989.

\bibitem[Lin18a]{Lin18I}
M.~Linckelmann.
\newblock {\em The block theory of finite group algebras. {V}ol. {I}},
  volume~91 of {\em London Mathematical Society Student Texts}.
\newblock Cambridge University Press, Cambridge, 2018.

\bibitem[Lin18b]{Lin18II}
M.~Linckelmann.
\newblock {\em The block theory of finite group algebras. {V}ol. {II}},
  volume~92 of {\em London Mathematical Society Student Texts}.
\newblock Cambridge University Press, Cambridge, 2018.

\bibitem[Mal08]{Mal08II}
G.~Malle.
\newblock The inductive {M}c{K}ay condition for simple groups not of {L}ie
  type.
\newblock {\em Comm. Algebra}, 36(2):455--463, 2008.

\bibitem[Mal14]{Mal14}
G.~Malle.
\newblock On the inductive {A}lperin--{M}c{K}ay and {A}lperin weight conjecture
  for groups with abelian {S}ylow subgroups.
\newblock {\em J. Algebra}, 397:190--208, 2014.

\bibitem[MNSFT22]{MNSFT}
G.~Malle, G.~Navarro, M.~Schaffer~Fry, and H.~T. Tiep.
\newblock Brauer's height zero conjecture.
\newblock {\em arXiv:2209.04736}, 2022.

\bibitem[MS16]{Mal-Spa16}
G.~Malle and B.~Sp{\"a}th.
\newblock Characters of odd degree.
\newblock {\em Ann. of Math. (2)}, 184(3):869--908, 2016.

\bibitem[NT89]{Nag-Tsu89}
H.~Nagao and Y.~Tsushima.
\newblock {\em Representations of finite groups}.
\newblock Academic Press, Inc., Boston, MA, 1989.

\bibitem[Nav98]{Nav98}
G.~Navarro.
\newblock {\em Characters and blocks of finite groups}, volume 250 of {\em
  London Mathematical Society Lecture Note Series}.
\newblock Cambridge University Press, Cambridge, 1998.

\bibitem[Nav18]{Nav18}
G.~Navarro.
\newblock {\em Character theory and the {M}c{K}ay conjecture}, volume 175 of
  {\em Cambridge Studies in Advanced Mathematics}.
\newblock Cambridge University Press, Cambridge, 2018.

\bibitem[NS14]{Nav-Spa14I}
G.~Navarro and B.~Sp{\"a}th.
\newblock On {B}rauer's height zero conjecture.
\newblock {\em J. Eur. Math. Soc. (JEMS)}, 16(4):695--747, 2014.

\bibitem[Ros22a]{Ros22}
D.~Rossi.
\newblock Character {T}riple {C}onjecture for {$p$}-solvable groups.
\newblock {\em J. Algebra}, 595:165--193, 2022.

\bibitem[Ros22b]{Ros-Thesis}
D.~Rossi.
\newblock {\em {C}haracter {T}riple {C}onjecture, towards the inductive
  condition for {D}ade's {C}onjecture for groups of {L}ie type}.
\newblock Dissertation, {B}ergische {U}niversit{\"a}t {W}uppertal, 2022.

\bibitem[Ros22c]{Ros-Clifford_automorphisms_HC}
D.~Rossi.
\newblock Inductive local-global conditions and generalized {H}arish-{C}handra
  theory.
\newblock {\em arXiv:2204.10301}, 2022.

\bibitem[Ros23a]{Ros-Homotopy}
D.~Rossi.
\newblock The {B}rown complex in non-defining characteristic and applications.
\newblock {\em arXiv:2303.13973}, 2023.

\bibitem[Ros23b]{Ros-Unip}
D.~Rossi.
\newblock A local-global principle for unipotent characters.
\newblock {\em arXiv:2301.05151}, 2023.

\bibitem[Ros23c]{Ros-iMcK}
D.~Rossi.
\newblock The {M}c{K}ay {C}onjecture and central isomorphic character triples.
\newblock {\em J. Algebra}, 618:42--55, 2023.

\bibitem[Ros24a]{Ros-Generalized_HC_theory_for_Dade}
D.~Rossi.
\newblock Counting conjectures and $e$-local structures in finite reductive
  groups.
\newblock {\em Adv. Math.}, 436:Paper No. 109403, 2024.

\bibitem[Ros24b]{Ros-Reduction_CTC}
D.~Rossi.
\newblock A reduction theorem for the {C}haracter {T}riple {C}onjecture.
\newblock {\em arXiv:2402.10632}, 2024.

\bibitem[Ruh22a]{Ruh22AM}
L.~Ruhstorfer.
\newblock The {A}lperin--{M}c{K}ay conjecture for the prime $2$.
\newblock {\em arXiv:2204.06373}, 2022.

\bibitem[Ruh22b]{Ruh22}
L.~Ruhstorfer.
\newblock Jordan decomposition for the {A}lperin--{M}c{K}ay conjecture.
\newblock {\em Adv. Math.}, 394:Paper No. 108031, 2022.

\bibitem[Ruh22c]{Ruh21}
L.~Ruhstorfer.
\newblock Quasi-isolated blocks and the {A}lperin--{M}c{K}ay conjecture.
\newblock {\em Forum Math. Sigma}, 10:Paper No. e48, 43, 2022.

\bibitem[Sp{\"a}12]{Spa12}
B.~Sp{\"a}th.
\newblock Inductive {M}c{K}ay condition in defining characteristic.
\newblock {\em Bull. Lond. Math. Soc.}, 44(3):426--438, 2012.

\bibitem[Sp{\"a}13]{Spa13II}
B.~Sp{\"a}th.
\newblock A reduction theorem for the {A}lperin--{M}c{K}ay conjecture.
\newblock {\em J. Reine Angew. Math.}, 680:153--189, 2013.

\bibitem[Sp{\"a}17]{Spa17}
B.~Sp{\"a}th.
\newblock A reduction theorem for {D}ade's projective conjecture.
\newblock {\em J. Eur. Math. Soc. (JEMS)}, 19(4):1071--1126, 2017.

\bibitem[Sp{\"a}18]{Spa18}
B.~Sp{\"a}th.
\newblock Reduction theorems for some global-local conjectures.
\newblock In {\em Local representation theory and simple groups}, EMS Ser.
  Lect. Math., pages 23--61. Eur. Math. Soc., Z{\"u}rich, 2018.

\end{thebibliography}

\vspace{1cm}

Mathematisches Forschungsinstitut Oberwolfach
\\
Schwarzwaldstr. 9-11,
\\
77709 Oberwolfach-Walke,
\\
Germany

\textit{Email address:} \href{mailto:damiano.rossi.math@gmail.com}{damiano.rossi.math@gmail.com}

\end{document}